\documentclass[11pt,a4 paper, reqno]{amsart}
\usepackage[utf8]{inputenc}
\usepackage[OT1]{fontenc}
\usepackage[english]{babel}
\usepackage{dsfont}
\usepackage{amsfonts}
\usepackage{amsmath}
\usepackage{bbm}
\usepackage{wasysym}
\usepackage{amsthm}
\usepackage{amssymb}
\usepackage{float}
\usepackage{textcomp}
\usepackage{xcolor}
\usepackage{graphicx}
\usepackage{fancybox}
\usepackage{fancyhdr}
\usepackage{mathrsfs}
\usepackage{tikz}
\usetikzlibrary{arrows}
\usetikzlibrary{calc}
\usepackage{centernot}
\usepackage{listings}
\usepackage{faktor}
\usepackage{bm}
\usepackage{setspace}
\usepackage{hyperref}
\usepackage{newlfont}
\usepackage{geometry}
\usepackage{ccaption}
\usepackage{tipa}
\usepackage{units}
\usepackage[autostyle,italian=guillemets]{csquotes}
\usepackage{comment}

\usepackage{guit}
\usepackage{tikz-cd}
\usepackage{enumerate}
\usepackage{enumitem}

\everymath{\displaystyle}

\theoremstyle{definition}
\newtheorem{Def}{Definition}[section]

\newtheorem{es}[Def]{Example}

\theoremstyle{remark}
\newtheorem{obs}[Def]{Remark}
\newtheorem{nota}[Def]{Notation}
\theoremstyle{plain}
\newtheorem{prop}[Def]{Proposition}

\newtheorem{lema}[Def]{Lemma}

\newtheorem{cor}[Def]{Corollary}

\newtheorem{teo}[Def]{Theorem}

\newtheorem*{theo:char}{Theorem~\ref{psi-duality}}

\newcommand{\bo}{\mathbf}
\newcommand{\A}{{\mathcal A}}
\newcommand{\B}{{\mathcal B}}
\newcommand{\C}{{\mathcal C}}
\newcommand{\D}{{\mathcal D}}
\newcommand{\E}{{\mathcal E}}
\newcommand{\F}{{\mathcal F}}

\newcommand{\I}{{\mathcal I}}

\newcommand{\K}{{\mathcal K}}

\renewcommand{\P}{{\mathcal P}}
\newcommand{\Q}{{\mathcal Q}}

\newcommand{\V}{{\mathcal V}}
\newcommand{\W}{{\mathcal W}}

\newcommand{\tx}{\textnormal}

\newcommand{\op}{^\textnormal{op}}

\newcommand*\cocolon{%
	\nobreak
	\mskip6mu plus1mu
	\mathpunct{}%
	\nonscript
	\mkern-\thinmuskip
	{:}%
	\mskip2mu
	\relax
}
\makeatother

\makeatletter
\newcommand{\changeoperator}[1]{%
	\csletcs{#1@saved}{#1@}%
	\csdef{#1@}{\changed@operator{#1}}%
}
\newcommand{\changed@operator}[1]{%
	\mathop{%
		\mathchoice{\textstyle\csuse{#1@saved}}
		{\csuse{#1@saved}}
		{\csuse{#1@saved}}
		{\csuse{#1@saved}}%
	}%
}
\makeatother

\makeatletter
\def\@tocline#1#2#3#4#5#6#7{\relax
	\ifnum #1>\c@tocdepth 
	\else
	\par \addpenalty\@secpenalty\addvspace{#2}%
	\begingroup \hyphenpenalty\@M
	\@ifempty{#4}{%
		\@tempdima\csname r@tocindent\number#1\endcsname\relax
	}{%
		\@tempdima#4\relax
	}%
	\parindent\z@ \leftskip#3\relax \advance\leftskip\@tempdima\relax
	\rightskip\@pnumwidth plus4em \parfillskip-\@pnumwidth
	#5\leavevmode\hskip-\@tempdima
	\ifcase #1
	\or\or \hskip 1em \or \hskip 2em \else \hskip 3em \fi%
	#6\nobreak\relax
	\hfill\hbox to\@pnumwidth{\@tocpagenum{#7}}\par
	\nobreak
	\endgroup
	\fi}
\makeatother

\changeoperator{sum}
\changeoperator{prod}
\changeoperator{coprod}

\title[Dualities in the theory of accessible categories]{Dualities in the theory of accessible categories}
\author{Giacomo Tendas}
\address{Department of Mathematics, University of Manchester, Alan Turing Building, Manchester, M13 9PL, UK}
\email{giacomo.tendas@manchester.ac.uk}
\date{\today}

\begin{document}

\maketitle

\begin{abstract}
	 Through the notion of {\em weakly sound} class of weights, we recover many known dualities involving accessible categories with a chosen class of limits, as instances of a general duality theorem. These include the Gabriel--Ulmer duality for locally finitely presentable categories, Diers duality for locally finitely multipresentable categories, and the Makkai--Par\'e duality for finitely accessible categories. In doing so, we extend these to the enriched setting, provide a more formal and unifying approach to the theory, and also discuss new dualities that arise as a consequence of our main theorem.
\end{abstract}

\tableofcontents

\section{Introduction}

A small category with finite limits can bee seen as a {\em theory}, whose models are the finite-limit preserving functors into $\bo{Set}$. These, from the model theoretic point of view, contain the algebraic and Horn theories, and are actually equivalent to the more general essentially algebraic theories \cite{coste}. Gabriel and Ulmer characterized the categories of models of these limit theories and called them {\em locally finitely presentable categories} \cite{GU71:libro}. 

This led to a duality between the 2-category of finite limit theories, on the one hand, and that of the locally finitely presentable categories on the other. These can equivalently be presented as those categories which admit all small limits and that arise as the free cocompletion of a small category under filtered colimits. A category satisfying the second condition is called {\em finitely accessible}. 

There are now multiple directions that one might take to generalize this duality. A first possibility is to replace finite limits with another class of limits $\Phi$, and then consider those complete categories which are freely generated under $\Phi$-filtered colimits (that is, those colimits that commute in $\bo{Set}$ with $\Phi$-limits). Of course some hypotheses on the class $\Phi$ are required. This approach was taken for instance in \cite{adamek2003duality} for the class of finite products, in \cite{centazzo2002duality} for any {\em (weakly) sound} class of limits $\Phi$ as in \cite{ABLR02:articolo}, and more generally in the 2-categorical context through the notion of GU envelope \cite{di2023accessibility}.

In this paper we shall follow a different direction which arises when, instead of replacing finite limits with another class, we relax the completeness condition on a locally finitely presentable category; so that we deal just with finitely accessible categories with a given class of limits. Also in this setting some dualities have been considered in the literature.

In the absence of any limit, Makkai and Par\'e characterized the 2-category determined by the finitely accessible categories as biequivalent to the dual of the 2-category of presheaf categories, lex left-adjoint functors, and natural transformations \cite{MP89:libro}. Similarly, Diers proved a duality in \cite{Die80:articolo} between the 2-category of locally finitely multipresentable categories (that is, finitely accessible with connected limits) and that of the finitely complete categories which are free cocompletions of a small category under coproducts. In each of these cases the duality theorem is induced by homming into $\bo{Set}$ from the ambient 2-category taken into consideration.

{\em Working with categories enriched over a locally presentable base $\V$, we capture all these dualities into a general framework by using the notion of weakly sound class of weights~\cite{ABLR02:articolo}. This will allow us to generalize most of the known dualities to the enriched context and at the same time to explore new cases not considered in the literature before.}

The {\em weak soundness} condition for a class of weights $\Psi$ (Definition~\ref{wsclass}) says, in a few words, that colimits weighted by any $\Psi$-continuous weights commute in $\V$ with $\Psi$-limits; a weight satisfying this property is called {\em $\Psi$-flat}. The prototype example of such a class is that of finite limits; this is why soundness was originally considered in~\cite{ABLR02:articolo}. There are of course many other examples, which are discussed below and, more in detail, in Example~\ref{examplewsc}. Unlike in the papers~\cite{ABLR02:articolo,LT22:virtual}, we do not assume $\Psi$ to be locally small and we do not consider the corresponding notion of {\em $\Psi$-accessibility}. Instead, we use $\Psi$ to indicate a class of limits that our accessible $\V$-categories are assumed to have, as it was done in~\cite{LT22:limits}.

Then, given a weakly sound class $\Psi$ and a regular cardinal $\alpha$, we prove the following theorem:

\begin{theo:char}
	The 2-adjunction 
	\begin{center}
		
		\begin{tikzpicture}[baseline=(current  bounding  box.south), scale=2]

			\node (f) at (0,0.4) {$\alpha\tx{-Acc}_\Psi(-,\V)\colon \alpha\tx{-}\bo{Acc}_\Psi$};
			\node (g) at (2.7,0.4) {$\Psi^+\tx{-}\bo{Free}_\alpha\op\cocolon \Psi^+\tx{-Ex}_\alpha(-,\V)$};
			
			\path[font=\scriptsize]

			([yshift=1.3pt]f.east) edge [->] node [above] {} ([yshift=1.3pt]g.west)
			([yshift=-1.3pt]f.east) edge [<-] node [below] {} ([yshift=-1.3pt]g.west);
		\end{tikzpicture}
		
	\end{center}
	is a biequivalence of 2-categories.
\end{theo:char}

Here, $\alpha\tx{-}\bo{Acc}_\Psi$ is the 2-category of $\alpha$-accessible $\V$-categories with $\Psi$-limits, $\Psi$-continuous and $\alpha$-flat colimit preserving $\V$-functors, and $\V$-natural transformations. On the other hand, ${\Psi^+}\tx{-} \bo{Free}_\alpha$ is defined as the 2-category whose objects are the $\alpha$-complete $\V$-categories that are free cocompletions of a small $\V$-category under $\Psi$-flat colimits, whose morphisms are the $\alpha$-continuous $\V$-functors that also preserve $\Psi$-flat colimits, and whose 2-cells are $\V$-natural transformations. All these notions will be introduced and explained in Section~\ref{duality1}.

When $\Psi=\P$ is the class of all small weights, then ${\Psi^+}$ is the class $\Q$ of the Cauchy (or absolute) weights, and we recover the Gabriel-Ulmer duality between locally $\alpha$-presentable categories and the small $\alpha$-complete ones, which in the enriched context was first proved by Kelly \cite{Kel82:articolo}. When $\Psi=\emptyset$ we obtain a duality between $\alpha$-accessible $\V$-categories and presheaves $\V$-categories, generalizing the result of \cite{MP89:libro} to the enriched context. If $\V=\bo{Set}$ and $\Psi$ is the class for connected limits, we recover Diers duality for locally finitely multipresentable categories \cite{Die80:articolo}. 

Finally, if $\Psi=\P_\alpha$ is the class of weights for $\alpha$-small limits, then the theorem translates into a bi-involution
$$ \alpha\tx{-}\bo{Acc}_\alpha \simeq (\alpha\tx{-}\bo{Acc}_\alpha)\op $$
of the 2-category of $\alpha$-accessible $\V$-categorias with $\alpha$-small limits (Corollary~\ref{bi-inv}). All these examples, and more, are discussed in Section~\ref{examples-duality}.

A more general duality theorem can be achieved using the theory of {\em companions} developed in~\cite{LT22:limits}; this is done in Chapter 6 of the author's PhD thesis~\cite{Ten22:phd}. In that more general framework one can also capture Hu's duality between the 2-category of {\em weakly locally finitely presentable} categories (that is, finitely accessible categories with products) and that of exact categories with enough projectives~\cite{Hu1992:book}; as well as Hu and Tholen's duality between the 2-category of {\em locally finitely polypresentable} categories (that is, finitely accessible categories with wide pullbacks) and that of finitely complete {\em quasi-free} categories~\cite{HT96qc:articolo}. The approach with companions, however, would make the theory much more technical (and triple the number of pages needed to prove it); therefore, since Theorem~\ref{psi-duality} is already quite general, it seemed best to treat the only the weakly sound case in this paper and redirect the reader to~\cite{Ten22:phd} for the more general setting.  

\noindent 
{\bf Outline.} In Section~\ref{gen-sett} we give the relevant background notions on enrichment, and we recall some results on free cocompletions, soundness, and accessibility. Next, in Section~\ref{duality1}, we introduce the 2-categories and the 2-functors involved in the duality theorem and then prove the main result in Theorem~\ref{psi-duality}.

Section~\ref{examples-duality} is devoted to examples; here we select specific weakly sound classes of weights and explain how the duality theorem can specialize to them. Finally, in Section~\ref{Scott-V} we extend the duality involving the 2-category of finitely accessible $\V$-categories, proven as Theorem~\ref{finacc-presh}, to a 2-adjunction between the 2-category of accessible $\V$-categories with filtered colimits and that of $\V$-topoi. 

\noindent
{\bf Acknowledgements.} I would like to thank Steve Lack for the many helpful comments and suggestions on a preliminary version of this paper, which ended up being a chapter of my PhD thesis~\cite{Ten22:phd}. Thanks also to Ivan Di Liberti for suggesting to generalize the Scott adjunction to the enriched context, and to the anonymous referee for making suggestions that improved the readability of the paper. Finally, I acknowledge with gratitude the support of the EPSRC, under the grant EP/X027139/1, and of the Grant Agency of the Czech Republic, under the grant 22-02964S.

\section{Our setting}\label{gen-sett}

\subsection{Background}$ $

Let us fix a base of enrichment $\V=(\V_0,\otimes,I)$ which is a complete, cocomplete, and symmetric monoidal closed category.

For matters concerning enrichment we follow the notations of \cite{Kel82:libro}, with the only change that ``indexed'' colimits are here called ``weighted'', as is now standard. Our $\V$-categories are allowed to have a large set of objects, unless specified otherwise.

Given a $\V$-category $\C$ and a $\V$-functor $F\colon\C\to\D$ we denote by $\C_0$ the underlying ordinary category of $\C$ and by $F_0\colon\C_0\to\D_0$ the underlying ordinary functor of $F$. This assignment extends to a 2-functor $(-)_0\colon\V$-$\bo{CAT}\to\bo{CAT}$ from the 2-category of $\V$-categories, $\V$-functors, and $\V$-natural transformations, to that of (ordinary) categories, functors, and natural transformations. 
This 2-functor has a left adjoint $(-)_\V\colon\bo{CAT}\to\V$-$\bo{CAT}$ that sends an ordinary category $\B$ to its {\em free $\V$-category} $\B_\V$; the universal property of the adjunction says that there is a natural isomorphism
$$ [\B_\V,\C]_0\cong [\B,\C_0] $$
for any $\V$-category $\C$.

We call {\em weight} a $\V$-functor $ M \colon \C\op\to\V$ with small domain. Given such a weight $ M $ and a $\V$-functor $H\colon \C\to\A$, we denote by $ M *H$ (if it exists) the colimit of $H$ weighted by $ M$. The universal property of such colimit is defined by isomorphisms
$$ \A(M*H,A)\cong [\C\op,\V](M,\A(H-,A)) $$
that are $\V$-natural in $A\in\A$.
Dually, given a weight $ N \colon \C\to\V$ and a $\V$-functor $K\colon \C\to\A$, the weighted limit of $K$ by $N$ is denoted by $\{ N ,K\}$. 

Conical limits and colimits are special cases of weighted ones; they coincide with those weighted by $\Delta I\colon \B_\V\op\to\V$ for some ordinary category $\B$. The conical colimit of a $\V$-functor $T_{\V}\colon \B_\V\to\A$, if it exists, will also be the ordinary colimit of the transpose $T\colon \B\to\A_0$ in $\A_0$ (but the converse is not generally true). 


For the remainder of the paper we assume in addition that $\V$ is {\em locally $\alpha$-presentable as a closed category}, meaning that it is locally $\alpha$-presentable and the full subcategory $\V_\alpha$ of the $\alpha$-presentable objects contains the unit and is closed under tensor product \cite{Kel82:articolo}. With this assumption the theory of accessibility extends to the enriched context \cite{BQ96:articolo}; see Section~\ref{accessible} below for more details.

\subsection{Free cocompletions and soundness}\label{free-sound-section}$ $

The enriched notion of free cocompletion under a class of weights $\Phi$ has been studied in \cite{KS05:articolo}; here we recall the definitions and some of the main properties.

Given a class $\Phi$ and a possibly large $\V$-category $\C$, the free cocompletion $\Phi(\C)$ of $\C$ can be constructed as the closure of the representables in $[\C\op,\V]$ under $\Phi$-colimits \cite[Section~3]{KS05:articolo}. While $[\C\op,\V]$ is not in general $\V$-enriched (but only $\V'$-enriched for some universe enlargement $\V'$ of $\V$), the resulting $\Phi(\C)$ is always a $\V$-category. 

By definition, $\Phi(\C)$ comes together with a fully faithful inclusion $Z\colon\C\hookrightarrow\Phi(\C)$; this satisfies the following universal property which makes $\Phi(\C)$ the {\em free cocompletion of $\C$ under $\Phi$-colimits}: for any $\Phi$-cocomplete $\K$, pre-composition by $Z$ induces an equivalence 
$$ \Phi\tx{-Coct}(\Phi(\C),\K)\simeq[\C,\K]; $$
here $\Phi\tx{-Coct}(\Phi(\C),\K)$ denotes the full subcategory of $[\Phi(\C),\K]$ spanned by the $\Phi$-cocontinuous $\V$-functors. The inverse of the equivalence above is given by left Kan extending along $Z$ (see \cite[Proposition~3.6]{KS05:articolo}).

\begin{obs}
	Of particular importance are the free cocompletion under the class $\P$ of all weights and the class $\Q$ of the Cauchy weights (those whose colimits commute with all limits in $\V$). Following~\cite{DL07}, given a $\V$-category $\C$, an element $F\in\P(\C)$ is a $\V$-functor $F\colon\C\op\to\V$ which is a (small) colimit of representables. Equivalently, $F\in\P(\C)$ if and only if it can be expressed as $F\cong\tx{Lan}_HFH$ where $H\colon\C'\hookrightarrow\C$ is a small full subcategory of $\C$. We refer to the objects of $\P(\C)$ as the {\em small presheaves} on $\C$.
\end{obs}

Dually, there is a notion of {\em free completion under $\Phi$-limits} which satisfies the dual universal property (with respect to $\Phi$-complete $\V$-categories and $\Phi$-continuous $\V$-functors). Given $\C$, we denote such completion by $\Phi^\dagger(\C)$. It follows from the universal property that for any $\V$-category $\C$ one has $\Phi^\dagger(\C)\simeq\Phi(\C\op)\op$.

It will be useful in the following sections to be able to characterize intrinsically those $\V$-categories which arise as free cocompletions under $\Phi$-colimits, and this is what the proposition below does. Given a $\V$-category $\B$ with $\Phi$-colimits, we denote by $\B_\Phi$ the full subcategory of $\B$ spanned by those objects $B$ for which $\B(B,-)$ is $\Phi$-cocontinuous. Then:

\begin{prop}[\cite{KS05:articolo}]\label{free-Phi}
	A $\V$-category $\B$ is equivalent to $\Phi(\C)$, for some small $\V$-category $\C$, if and only if $\B$ is $\Phi$-cocomplete and there exists a small full subcategory $\D$ of $\B_\Phi$ whose closure in $\B$ under $\Phi$-colimits is $\B$ itself. Moreover, in this case $\B_\Phi\simeq\Q(\D)$, is the Cauchy completion of $\D$, and is therefore (essentially) small.
\end{prop}
\begin{proof}
	Put together Propositions~4.3 and 7.5 of \cite{KS05:articolo}.
\end{proof}

Next we recall the notion of (possibly large) weakly sound class of weights considered in \cite[Section~4.1]{LT22:limits}, which generalizes the original notion of~\cite{ABLR02:articolo}. 

Recall from~\cite{LT22:virtual} that a $\V$-category $\A$ is called {\em virtually cocomplete} if its free completion $\P^\dagger(\A)$ under all small limits is cocomplete. Equivalently, by~\cite[Theorem~3.8]{DL07} and using that $\P^\dagger(\A)\subseteq[\A,\V]\op$, $\A$ is virtually cocomplete if any small limit of representables in $[\A,\V]$ is a small $\V$-functor $\A\to\V$.

\begin{Def}[\cite{LT22:limits}]\label{wsclass}
	Let $\Psi$ be a class of weights. We say that a small presheaf $M\colon\A\op\to\V$ is {\em $\Psi$-flat} if $M$-colimits commute in $\V$ with $\Psi$-limits. We denote by $\Psi^+$ the class given by those $\Psi$-flat $\V$-functors with small domain.\\
	A class of weights $\Psi$ is called {\em weakly sound} if every $\Psi$-continuous and small $\V$-functor $ M \colon\A\to\V$, from a virtually cocomplete and $\Psi$-complete $\A$, is $\Psi$-flat.
\end{Def}

The virtual cocompleteness of $\A$ is asked since in \cite{LT22:limits} the condition above is applied to the case where $\A$ is an accessible $\V$-category (see Section~\ref{accessible} below) which is in particular virtually cocomplete (Proposition~\ref{virtual}). 

Here are some examples weakly sound classes of weights:

\begin{es}\label{examplewsc}
	In the following list we give
	examples of weakly sound classes $\Psi$; when meaningful, we explicitly state what $\Psi^+$ looks like. See \cite[Example~4.8]{LT22:limits} for details. In cases where
	$\V=\bo{Set}$ and we speak of a class of categories rather than a class of weights, we are referring to the corresponding conical limits (in other words, weighted by the terminal weight). 
	\begin{enumerate}[align=parleft,left=10pt]
		\item $\Psi=\P$ is the class of all small weights. Then $\Psi^+=\Q$ is the class of the Cauchy (or absolute) weights \cite{KS05:articolo}.
		\item $\Psi=\emptyset$. Then $\Psi^+=\P$ is the class of all weights. 
		\item $\V$ locally $\alpha$-presentable as a closed category, $\Psi$ is the class of $\alpha$-small weights (see Definition~\ref{alpha-small}). Then $\Psi^+$ consists of the $\alpha$-flat $\V$-functors (see Section~\ref{accessible} below). For $\V=\bo{Set}, \bo{Cat},\bo{Pos},\bo{SSet}$, or any other $\V$ as in \cite[Section~3]{LT21:articolo}, $\alpha$-flat colimits are the same as (conical) $\alpha$-filtered colimits. This is not true for a general $\V$.
		\item $\V$ symmetric monoidal closed finitary quasivariety as in \cite{LT20:articolo}, $\Psi$ is the class for finite products and finitely presentable projective powers. Examples of such $\V$ are the categories $\bo{Ab}$ of abelian groups, $\bo{GAb}$ of graded abelian groups, and $\bo{DGAb}$ of differentially graded abelian groups.
		\item $\V$ cartesian closed, $\Psi$ is the class of finite discrete diagrams. Then $\Psi^+$ consists of those $M$ for which $\tx{Lan}_\Delta M \cong M \times M $ (\cite[Lemma~2.3]{KL93FfKe:articolo}). Examples of such $\V$ are the categories $\bo{Pos}$ of partially ordered sets, $\bo{Cat}$ of small categories, and $\bo{SSet}$ of simplicial sets.
		\item $\V=\bo{Set}$, $\Psi$ is the class of connected categories. Then $\Psi^+$ is generated by the class of discrete categories.
		\item $\V=\bo{Set}$, $\Psi=\{\emptyset\}$. Then $\Psi^+$ is generated by the class of connected categories.
		\item $\V=\bo{Set}$, $\Psi$ consists of the finite connected categories. Then $\Psi^+$ is generated by coproducts of filtered categories.
		\item $\V=\bo{Set}$, $\Psi$ is the class of finite non empty categories. Then $\Psi^+$ is generated by the filtered categories plus the empty category.
		\item $\V=\bo{Set}$, $\Psi$ is the class of finite discrete non empty categories. Then $\Psi^+$ is generated by the sifted categories plus the empty category.
		\item $\V=\bo{Cat}$, $\Psi$ is the class generated by connected conical limits and powers by connected categories. Then every $\Psi$-flat weight is a split subobject of coproducts of representables.
	\end{enumerate}
\end{es}

Next we prove an easy generalization of \cite[Lemma~3.3]{LT22:limits} since it will be useful several times in our proofs. 

\begin{lema}\label{flat-restriction}
	Let $\Psi$ be a class of weights, $J\colon \B\to\C$ be a $\V$-functor with $\B$ a small $\V$-category, and $  M  \colon \B\to\V$ a weight. Then:\begin{enumerate}\setlength\itemsep{0.25em}
		\item if $  M  $ is $\Psi$-flat then so is $\tx{Lan}_{J}  M $;
		\item if $J$ is fully faithful and $\tx{Lan}_{J} M  $ is $\Psi$-flat then $ M $ is $\Psi$-flat as well;
	\end{enumerate}
\end{lema}
\begin{proof}
	By definition, a $\V$-functor $ M $ is $\Psi$-flat if $ M *-\colon [\B\op,\V]\to\V$ is $\Psi$-continuous. Note that the triangle below commutes up to isomorphism.
	\begin{center}
		\begin{tikzpicture}[baseline=(current  bounding  box.south), scale=2]
			
			\node (a) at (0.7,0.6) {$[\B\op,\V]$};
			\node (c) at (0, 0) {$[\C\op,\V]$};
			\node (d) at (1.6, 0) {$\V$};
			
			\path[font=\scriptsize]
			
			(c) edge [->] node [above] {$[J\op,\V]\ \ \ \ \ \ \ \ \ \ $} (a)
			(a) edge [->] node [above] {$\ \ \ \ \ \  M *-$} (d)
			(c) edge [->] node [below] {$(\tx{Lan}_{J}M)*-$} (d);
			
		\end{tikzpicture}	
	\end{center} 
	Indeed, we need to check that $(\tx{Lan}_{J}M)*F\cong  M *FJ\op$, $\V$-naturally in $F\colon\C\op\to\V$. This follows from \cite[4.19]{Kel82:libro} or simply by considering the following chain of isomorphisms
	\begin{align}
		[(\tx{Lan}_{J}M)*F,X]&\cong [\C,\V](\tx{Lan}_{J}M,[F-,X])\tag{def of $(\tx{Lan}_{J^{op}}M)*F$}\\
		 &\cong [\B,\V](M,[FJ\op-,X])\tag{$(\tx{Lan}_{J}(-)\dashv (-)\circ J$}\\
		 &\cong [M*FJ,X]\tag{def of $M*FJ$}
	\end{align}
	which are $\V$-natural in $X\in\V$.
	
	Note that, since $\C$ may not be small, above we are seeing $[\C\op,\V]$ as a $\V'$-category for some universe enlargement $\V'$ of $\V$ as in \cite[Section~3.12]{Kel82:libro}.
	
	As a consequence if $ M $ is $\Psi$-flat then so is $\tx{Lan}_{J} M$ since $[J\op,\V]$ is continuous; this proves (1).
	
	(2). If $J$ is fully faithful and $\tx{Lan}_{J}M $ is $\Psi$-flat, then 
	\begin{equation*}
		\begin{split}
			M *-&\cong ( M *-)\circ id_{[\B\op,\V]} \\
			&\cong ( M *-)\circ [J\op,\V]\circ \tx{Ran}_{J\op}\\
			&\cong (\tx{Lan}_{J}M*-)\circ \tx{Ran}_{J\op}\\
		\end{split}
	\end{equation*}
	where $id_{[\B\op,\V]}\cong[J\op,\V]\circ \tx{Ran}_{J\op}$ since $J$ is fully faithful, and $\tx{Ran}_{J\op}$ exists since the domain of $J$ is small. It follows that $ M *-$ is $\Psi$-continuous because $\tx{Lan}_{J}M*-$ is and $\tx{Ran}_{J\op}$ is continuous.
\end{proof}

\subsection{Accessible $\V$-categories}\label{accessible}$ $

Enriched accessible categories will be the main objects involved in the duality theorem that we aim to prove. While the concept in the ordinary context is very well studied and understood~\cite{MP89:libro,AR94:libro}, in the enriched setting they first appeared only at the end ot the 90s in~\cite{BQ96:articolo,BQR98}. Afterwards, they were considered again in \cite{LT21:articolo,LT22:virtual} where additional characterization theorems were proved, making the theory more tractable. In this section we recall their definition as well as some results that will be needed later.

The notions of $\alpha$-small and $\alpha$-flat weights, as well as those of limits and colimits weighted by them, are central when working with accessible $\V$-categories, so we recall their definitions below.

\begin{Def}[\cite{Kel82:articolo}]\label{alpha-small}
	A weight $ M \colon \C\op\to\V$ is {\em $\alpha$-small} if $\C$ has less than $\alpha$ objects, $\C(C,D)\in\V_{\alpha}$ for any $C,D\in\C$, and $ M (C)\in\V_\alpha$ for any $C\in\C$. An $\alpha$-small (weighted) limit is one taken along an $\alpha$-small weight. We say that a $\V$-category $\C$ is $\alpha$-complete if it has all $\alpha$-small limits; a $\V$-functor is $\alpha$-continuous if it preserves all $\alpha$-small limits. We denote by $\P_\alpha$ the class of $\alpha$-small weights.
\end{Def}

Both conical $\alpha$-small limits and powers by $\alpha$-presentable objects are examples of $\alpha$-small limits and together they generate all $\alpha$-small weighted limits \cite[Section~4]{Kel82:articolo}.

\begin{obs}\label{absolute-alpha}
	Note that every small $\alpha$-complete $\V$-category is Cauchy complete. Indeed, let $\C$ be small and $\alpha$-complete; then by the enriched Gabriel-Ulmer duality \cite{Kel82:articolo}, $\C\simeq\K_\alpha\op$ is equivalent to the opposite of the $\V$-category spanned by the $\alpha$-presentable objects of a locally $\alpha$-presentable $\K$, and this is Cauchy complete by \cite[Proposition~3.9]{LT22:virtual}.
\end{obs}

\begin{Def}[\cite{Kel82:articolo}]
	A weight $ M \colon\C\op\to\V$ is $\alpha$-flat if it is $\P_\alpha$-flat (in the sense of Definition~\ref{wsclass}); that is, if $M$-weighted colimits commute in $\V$ with $\alpha$-small limits.  By {\em $\alpha$-flat colimits} we mean those colimits weighted by an $\alpha$-flat weight.
\end{Def}

Note that every conical $\alpha$-filtered colimit is $\alpha$-flat and that the $\alpha$-flat $\V$-functors are closed in $[\C\op,\V]$ under $\alpha$-flat colimits.

\begin{prop}[\cite{Kel82:articolo}]\label{flat-char}
	Let $ M \colon\C\op\to\V$ be a weight; the following are equivalent:\begin{enumerate}\setlength\itemsep{0.25em}
		\item $ M $ is $\alpha$-flat;
		\item $ M $ is an $\alpha$-flat colimit of representables.
	\end{enumerate}
	If $\C$ is $\alpha$-cocomplete they are further equivalent to:\begin{enumerate}\setlength\itemsep{0.25em}
		\item[(3)] $ M $ is $\alpha$-continuous;
		\item[(4)] $ M $ is a conical $\alpha$-filtered colimit of representables. 
	\end{enumerate}
\end{prop}

In general, when $\C$ is not $\alpha$-cocomplete, not every $\alpha$-flat weight can be expressed as an $\alpha$-filtered colimit of representables, see \cite{LT21:articolo}. 

\begin{obs}\label{Phi-accessible}
	It follows from Proposition~\ref{flat-char} above, that the class $\P_\alpha$ is weakly sound; this is what makes the theory of \cite{BQR98,LT22:virtual} work. In fact, in \cite[Section~3]{LT22:virtual}, it is considered a notion of $\Phi$-accessibility for a {\em locally small} and weakly sound class of weights $\Phi$, generalizing the work of~\cite{ABLR02:articolo}. (Local smallness of $\Phi$ here means that $\Phi(\C)$ is small whenever $\C$ is.)
	
	The notion of $\Phi$-accessibility could replace that of $\alpha$-accessibility everywhere in this paper (see Remark~\ref{moregeneral-sound}); we decided to keep the simpler notion of $\alpha$-accessibility since the readers will be more familiar with that. 
\end{obs}

\begin{Def}[\cite{BQR98}]
	Let $\alpha$ be a regular cardinal and $\A$ be a $\V$-category with $\alpha$-flat colimits.  An object $A\in\A$ is called {\em $\alpha$-presentable} if $\A(A,-)$ preserves $\alpha$-flat colimits. We denote by $\A_\alpha$ the full subcategory of $\A$ spanned by the $\alpha$-presentable objects. 
\end{Def}	

\begin{obs}
	Following the notation of Section~\ref{free-sound-section}, note that the $\V$-category $\A_\alpha$ coincides with $\A_{\P_\alpha^+}$. 
\end{obs}
	
\begin{Def}[\cite{BQR98}]
	We say that a $\V$-category $\A$ with $\alpha$-flat colimits is {\em $\alpha$-accessible} if there exists $\C\subseteq\A_\alpha$ small for which every element of $\A$ is an $\alpha$-flat colimit of objects from $\C$. We denote by $\alpha\tx{-}\bo{Acc}$ the 2-category of $\alpha$-accessible $\V$-categories, $\V$-functors preserving $\alpha$-flat colimits, and $\V$-natural transformations.
\end{Def}

Then $\alpha$-accessible $\V$-categories can be characterized as follows:

\begin{prop}\cite[Proposition~3.10]{LT22:virtual}\label{acc=flat}
	The following are equivalent for a $\V$-category $\A$: \begin{enumerate}\setlength\itemsep{0.25em}
		\item $\A$ is $\alpha$-accessible;
		\item $\A$ is the free cocompletion of a small $\V$-category $\C$ under $\alpha$-flat colimits;
		\item $\A\simeq\alpha\tx{-}\tx{Flat}(\C^{op},\V)$ for some small $\C$.
	\end{enumerate}
	In both $(2)$ and $(3)$ the $\V$-category $\C$ can be chosen to be $\A_\alpha$; moreover, if $\C$ is Cauchy complete, then $\alpha\tx{-}\tx{Flat}(\C^{op},\V)_\alpha\simeq \C$.
\end{prop}

Lastly, we shall need the following result. Recall that a $\V$-category is called {\em accessible} if it is $\alpha$-accessible for some $\alpha$.

\begin{prop}\cite{LT22:virtual}\label{virtual}
	Let $\A$ be an accessible $\V$-category; then $\P^\dagger(\A)=\P(\A\op)\op$ consists of those $\V$-functors $\A\to\V$ which preserve $\alpha$-flat colimits for some $\alpha$. In particular, $\P(\A\op)$ is complete and closed in $[\A,\V]$ under small limits, making $\A$ virtually cocomplete.
\end{prop}
\begin{proof}
	Put together Propositions~4.9 and 4.18 from~\cite{LT22:virtual}.
\end{proof}

\section{The duality}\label{duality1}

Now we are ready to define the 2-categories involved in the duality theorem.

\begin{Def}
	Let $\alpha\tx{-}\bo{Acc}_\Psi$ be the 2-category with objects the $\alpha$-accessible $\V$-categories with $\Psi$-limits, $\Psi$-continuous $\V$-functors which preserves $\alpha$-flat colimits as morphisms, and $\V$-natural transformations as 2-cells.
\end{Def}

On the other hand consider the following:

\begin{Def}
	We say that a $\V$-category $\B$ is {\em $\Psi^+$-free} if it is the free cocompletion of a small $\V$-category under $\Psi^+$-colimits. Let ${\Psi^+}\tx{-}\bo{Free}_\alpha$ be the 2-category of the $\alpha$-complete and $\Psi^+$-free $\V$-categories, $\alpha$-continuous and $\Psi^+$-cocontinuous $\V$-functors, and $\V$-natural transformations.
\end{Def}

By Proposition~\ref{free-Phi}, a $\V$-category with $\Psi^+$-colimits $\B$ is $\Psi^+$-free if and only if $\B_{\Psi^+}$ is essentially small and its closure in $\B$ under $\Psi^+$-colimits is $\B$ itself.

\begin{nota}\label{notation}
	Given $\A\in \alpha\tx{-}\bo{Acc}_\Psi$, we denote with $\alpha\tx{-}\bo{Acc}_\Psi(\A,\V)$ --- in bold --- the hom in the 2-category $\alpha\tx{-}\bo{Acc}_\Psi$. This is the (ordinary) category of $\alpha$-accessible and $\Psi$-continuous $\V$-functors $\A\to\V$, and $\V$-natural transformations between them. Instead, we denote by $\alpha\tx{-Acc}_\Psi(\A,\V)$ --- unbolded --- the $\V$-category obtained by taking the full subcategory of $[\A,\V]$ spanned by the $\alpha$-accessible and $\Psi$-continuous $\V$-functors (the fact that this is actually a $\V$-category, and not just a $\V'$-category for some universe enlargement $\V'\supseteq\V$, follows from Lemma~\ref{2-functor-2} below). Therefore we obtain an equality
	$$\alpha\tx{-Acc}_\Psi(\A,\V)_0=\alpha\tx{-}\bo{Acc}_\Psi(\A,\V).$$
	We use the same notational convention for $\B\in{\Psi^+}\tx{-}\bo{Free}_\alpha$, so that $\Psi^+\tx{-Free}_\alpha(\B,\V)$ denotes the $\V$-category obtained by taking the full subcategory of $[\B,\V]$ spanned by the $\alpha$-continuous and $\Psi^+$-cocontinuous $\V$-functors (again, this will actually be a $\V$-category by Lemma~\ref{2-functor-1}). And therefore we obtain an equality
	$$ \Psi^+\tx{-Free}_\alpha(\B,\V)_0={\Psi^+}\tx{-}\bo{Free}_\alpha(\B,\V) $$
	of ordinary categories. 
\end{nota}

Next we define the 2-functors which will form the duality theorem.

\begin{lema}\label{2-functor-2}
	The 2-functor
	$$ \alpha\tx{-Acc}_\Psi(-,\V)\colon\alpha\tx{-}\bo{Acc}_\Psi\longrightarrow \Psi^+\tx{-}\bo{Free}_\alpha\op$$%
	is well-defined. Moreover, for any $\A\in\alpha\tx{-}\bo{Acc}_\Psi$ we have an equivalence
	$$\alpha\tx{-Acc}_\Psi(\A,\V)\simeq \Psi^+(\A_\alpha\op),$$ where we see $\A_\alpha\op$ as a full subcategory of $\alpha\tx{-Acc}_\Psi(\A,\V)$ through the Yoneda embedding.
\end{lema}
\begin{proof}
	We first prove the second part of the statement. Let $\A$ be $\alpha$-accessible with $\Psi$-limits; then the domain and codomain restriction of the Yoneda embedding induces a fully faithful inclusion
	$$ J\colon\A_\alpha\op\hookrightarrow \alpha\tx{-Acc}_\Psi(\A,\V)$$
	since for each $A\in\A_\alpha$ the $\V$-functor $\A(A,-)$ preserves $\Psi$-limits and $\alpha$-flat colimits. Note also that $\alpha\tx{-Acc}_\Psi(\A,\V)$ is closed in $\P(\A\op)$ under $\Psi^+$-colimits and $\alpha$-small limits since these commute in $\V$ with $\Psi$-limits and $\alpha$-flat colimits, and all these limits and colimits are computed pointwise in $\P(\A\op)$ by Proposition~\ref{virtual}. As a consequence, for any $A\in\A_\alpha$ the hom $\V$-functor $\alpha\tx{-Acc}_\Psi(\A,\V)(JA,-)$ preserves $\Psi^+$-colimits (it being the composite of the inclusion $\alpha\tx{-Acc}_\Psi(\A,\V)\to\P(\A\op)$ and of the evaluation $\V$-functor at $A$).
	
	Then, given any $F\in\alpha\tx{-Acc}_\Psi(\A,\V)$, since $F$ preserves $\alpha$-flat colimits it is the left Kan extension of its restriction to $\A_\alpha$ (universal property of the free cocompletion); that is, $F\cong\tx{Lan}_H(FH)$ where $H\colon \A_\alpha\to\A$ is the inclusion. Note that $FH$ is $\Psi$-flat by Lemma~\ref{flat-restriction} since $F$ was (being $\Psi$-continuous).
	
	Now, let $K\colon \alpha\tx{-Acc}_\Psi(\A,\V)\hookrightarrow\P(\A\op)$ be the inclusion; then $KJ\cong YH$ where $Y\colon \A\op\hookrightarrow\P(\A\op)$ is the Yoneda embedding. It follows that
	$$KF\cong F*Y\cong \tx{Lan}_H(FH)*Y\cong FH*YH\cong FH*KJ,$$
	where the first isomorphism is given by \cite[3.17]{Kel82:libro} and the third by \cite[4.19]{Kel82:libro}. Since $K$ is fully faithful, this implies that:
	$$ F\cong FH*J.$$
	Thus every element of $\alpha\tx{-Acc}_\Psi(\A,\V)$ is a $\Psi^+$-colimit of elements from $\A_\alpha\op$. This implies that $\alpha\tx{-Acc}_\Psi(\A,\V)\simeq\Psi^+(\A_\alpha\op)$ by Proposition~\ref{free-Phi}.
	
	To conclude that the 2-functor is well-defined we only need to check that for any morphism $F$ in $\alpha\tx{-Acc}_\Psi(\A,\V)$, the $\V$-functor $\alpha\tx{-Acc}_\Psi(F,\V)$ preserves $\Psi^+$-colimits and $\alpha$-small limits. But this follows from the fact that $\alpha\tx{-Acc}_\Psi(\A,\V)$ is closed in $\P(\A\op)$ under $\Psi^+$-colimits and $\alpha$-small limits and that precomposition by $F$ in the presheaf $\V$-categories is continuous and cocontinuous.
\end{proof}

The second 2-functor involved in the duality is given below:

\begin{lema}\label{2-functor-1}
	The 2-functor
	$$ \Psi^+\tx{-Free}_\alpha(-,\V)\colon \Psi^+\tx{-}\bo{Free}_\alpha\op\longrightarrow \alpha\tx{-}\bo{Acc}_\Psi$$%
	is well-defined. Moreover, for any $\B$ in $\Psi^+\tx{-}\bo{Free}_\alpha$ we have an equivalence $$\Psi^+\tx{-Free}_\alpha(\B,\V)\simeq\alpha\tx{-Flat}(\B_{\Psi^+},\V)$$ induced by pre-composition with the inclusion $\B_{\Psi^+}\hookrightarrow\B$.
\end{lema}
\begin{proof}
	Let $\B$ be in $\Psi^+\tx{-}\bo{Free}_\alpha$ and denote by $H\colon\B_{\Psi^+}\hookrightarrow\B$ the inclusion. By the universal property of the free cocompletion under $\Psi^+$-colimits, left Kan extending along $H$ induces an equivalence
		$$ [\B_{\Psi^+},\V]\simeq \Psi^+\tx{-Coct}(\B,\V). $$
	Now, by Lemma~\ref{flat-restriction} a $\V$-functor $F\colon\B_{\Psi^+}\to\V$ is $\alpha$-flat if and only if $\tx{Lan}_HF$ is, if and only if $\tx{Lan}_HF$ is $\alpha$-continuous (by weak soundness of $\P_\alpha$). Thus, the equivalence above restricts to
	  $$ \alpha\tx{-Flat}(\B_{\Psi^+},\V)\simeq \Psi^+\tx{-Free}_\alpha(\B,\V) $$
	as claimed in the final part of the statement.
	
	This proves that for any $\B$ as above the $\V$-category $\Psi^+\tx{-Free}_\alpha(\B,\V)$ is $\alpha$-accessible (Proposition~\ref{acc=flat}); moreover, since $\Psi$-limits and $\alpha$-flat colimits commute in $\V$ with $\Psi^+$-colimits and $\alpha$-small limits, $\Psi^+\tx{-Free}_\alpha(\B,\V)$ is closed in the presheaves under such limits and colimits; thus $\Psi^+\tx{-Free}_\alpha(\B,\V)$ lies in $\alpha\tx{-}\bo{Acc}_\Psi$. 
	
	The closure of $\Psi^+\tx{-Free}_\alpha(\B,\V)$ under $\Psi$-limits and $\alpha$-flat colimits also implies that precomposition with $F\colon \B\to\B'$ in $\Psi^+\tx{-}\bo{Free}_\alpha$ induces a $\V$-functor $\Psi^+\tx{-Free}_\alpha(F,\V)$ that preserves $\alpha$-flat colimits and $\Psi$-limits. Thus the 2-functor is well-defined.
\end{proof}

Note that the two 2-functors just defined assemble into a 2-adjunction

\begin{center}
	
	\begin{tikzpicture}[baseline=(current  bounding  box.south), scale=2]

		\node (f) at (0,0.4) {$\alpha\tx{-Acc}_\Psi(-,\V)\colon \alpha\tx{-}\bo{Acc}_\Psi$};
		\node (g) at (2.7,0.4) {$\Psi^+\tx{-}\bo{Free}_\alpha\op\cocolon \Psi^+\tx{-Free}_\alpha(-,\V)$.};
		
		\path[font=\scriptsize]

		([yshift=-1.3pt]f.east) edge [->] node [above] {} ([yshift=-1.3pt]g.west)
		([yshift=1.3pt]f.east) edge [<-] node [below] {} ([yshift=1.3pt]g.west);
	\end{tikzpicture}
	
\end{center}
where $\alpha\tx{-Acc}_\Psi(-,\V)\dashv \Psi^+\tx{-Free}_\alpha(-,\V)$. Indeed, this is the same as asking for an isomorphism of categories
$$ \alpha\tx{-}\bo{Acc}_\Psi(\A,\Psi^+\tx{-Free}_\alpha(\B,\V))\cong \Psi^+\tx{-}\bo{Free}_\alpha(\B,\alpha\tx{-Acc}_\Psi(\A,\V)). $$
that is natural in $\A\in \alpha\tx{-}\bo{Acc}_\Psi$ and $\B\in \Psi^+\tx{-}\bo{Free}_\alpha$. And this follows from the fact each category above is isomorphic to the category of $\V$-functors $\A\otimes\B\to \V$ which preserve $\alpha$-flat colimits and $\Psi$-limits in the first variable, and $\Psi^+$-colimits and $\alpha$-small limits in the second (here $\otimes$ is the tensor product of $\V$-categories as in~\cite[Section~2.3]{Kel82:libro}).

Finally, we can prove the following duality theorem which, as we will see in the next section, captures the known dualities for locally finitely presentable, multipresentable, and accessible categories as instances of the same theory.

\begin{teo}\label{psi-duality}
	The 2-adjunction 
	\begin{center}
		
		\begin{tikzpicture}[baseline=(current  bounding  box.south), scale=2]

			\node (f) at (0,0.4) {$\alpha\tx{-Acc}_\Psi(-,\V)\colon \alpha\tx{-}\bo{Acc}_\Psi$};
			\node (g) at (2.7,0.4) {$\Psi^+\tx{-}\bo{Free}_\alpha\op\cocolon \Psi^+\tx{-Ex}_\alpha(-,\V)$};
			
			\path[font=\scriptsize]

			([yshift=-1.3pt]f.east) edge [->] node [above] {} ([yshift=-1.3pt]g.west)
			([yshift=1.3pt]f.east) edge [<-] node [below] {} ([yshift=1.3pt]g.west);
		\end{tikzpicture}
		
	\end{center}
	is a biequivalence of 2-categories.
\end{teo}
\begin{proof}
	It is enough to prove that the unit and counit of the 2-adjunction are equivalences.\\
	Consider $\B\in \Psi^+\tx{-}\bo{Free}_\alpha$, then the unit at $\B$ is given by the evaluation $\V$-functor 
	$$\tx{ev}\colon  \B\to \alpha\tx{-Acc}_\Psi(\Psi^+\tx{-Free}_\alpha(\B,\V),\V) ;$$
	this maps $B\in \B$ to $\tx{ev}_B\colon \Psi^+\tx{-Free}_\alpha(\B,\V)\to\V$ sending $F$ to $\tx{ev}_B(F):=FB$.
	
	By Lemma~\ref{2-functor-1}, we know that $\Psi^+\tx{-Free}_\alpha(\B,\V)\simeq \alpha\tx{-Flat}(\B_{\Psi^+},\V)$; therefore, by Lemma~\ref{2-functor-2} and the fact that $\alpha\tx{-Flat}(\B_{\Psi^+},\V)_\alpha\simeq\B_{\Psi^+}\op$, we obtain
	$$ \alpha\tx{-Acc}_\Psi(\alpha\tx{-Flat}(\B_{\Psi^+},\V),\V)\simeq \Psi^+(\alpha\tx{-Flat}(\B_{\Psi^+},\V)_\alpha\op)\simeq \Psi^+(\B_{\Psi^+})\simeq \B$$
	where the last holds by definition of $\B$. It is easy to see that the composite of this equivalences, from left to right, is an inverse to the evaluation $\V$-functor.
	
	The proof that the counit is an equivalence is quite similar. Given $\A\in \alpha\tx{-}\bo{Acc}_\Psi$ the counit at $\A$ is given by the evaluation $\V$-functor 
	$$\tx{ev}'\colon  \A\to \Psi^+\tx{-Free}_\alpha(\alpha\tx{-Acc}_\Psi(\A,\V),\V) $$
	which can be described as above.
	Then, by Lemma~\ref{2-functor-2} we know that $\alpha\tx{-Acc}_\Psi(\A,\V)\simeq \Psi^+(\A_\alpha\op)$; thus, applying also Lemma~\ref{2-functor-1}, we obtain
	$$ \Psi^+\tx{-Free}_\alpha(\Psi^+(\A_\alpha\op),\V)\simeq \alpha\tx{-Flat}(\Psi^+(\A_\alpha\op)_{\Psi^+},\V)\simeq\alpha\tx{-Flat}(\A_\alpha\op,\V)\simeq \A. $$
	Again, the composite from left to right is easily seen to be an inverse to the evaluation $\V$-functor.
\end{proof}

\begin{obs}
	It is actually possible to prove a slightly stronger statement than that given in Theorem~\ref{psi-duality} above. Denote by $\W=(\V\tx{-}\bo{CAT},\otimes,\I)$ the monoidal category given by $\V\tx{-}\bo{CAT}$ together with the tensor product defined in \cite[Section~2.3]{Kel82:libro}.
	
	We believe it is possible to show (but we do not do this here) that both $\alpha\tx{-}\bo{Acc}_\Psi$ and $\Psi^+\tx{-}\bo{Free}_\alpha$ can be endowed with a structure of $\W$-category by taking the hom-objects to be the $\V$-categories $\alpha\tx{-}\tx{Acc}_\Psi(\A,\A')$ and $\Psi^+\tx{-Free}_\alpha(\B,\B')$ respectively. Assuming this, the proof of Theorem~\ref{psi-duality} actually shows that we have a biequivalence of $\W$-categories. Then, the biequivalence of 2-categories above could be obtained by applying the change of base functor $(-)_0\colon\W=\V\tx{-}\bo{CAT}\to\bo{CAT}$.
\end{obs}

\begin{obs}\label{moregeneral-sound}
	As mentioned in Remark~\ref{Phi-accessible}, Theorem~\ref{psi-duality} could be further generalized to the setting of a locally small and weakly sound class $\Phi$ in place of the class of $\alpha$-small weights. The idea being that all the proofs can be adapted by replacing $\alpha$-small limits with $\Phi$-limits and $\alpha$-flat colimits with $\Phi$-flat colimits, and by using the soundness property, as well as Lemma~\ref{flat-restriction} (applied to $\Phi$). We decided to keep the case of $\alpha$-small limits as the default setting to avoid an even heavier notation.\\
	The resulting theorem then would say that the 2-functors 
	\begin{center}
		
		\begin{tikzpicture}[baseline=(current  bounding  box.south), scale=2]

			\node (f) at (0,0.4) {$\Phi\tx{-Acc}_\Psi(-,\V)\colon \Phi\tx{-}\bo{Acc}_\Psi$};
			\node (g) at (2.7,0.4) {$\Psi^+\tx{-}\bo{Free}_\Phi\op\cocolon \Psi^+\tx{-Free}_\Phi(-,\V)$};
			
			\path[font=\scriptsize]

			([yshift=1.3pt]f.east) edge [->] node [above] {} ([yshift=1.3pt]g.west)
			([yshift=-1.3pt]f.east) edge [<-] node [below] {} ([yshift=-1.3pt]g.west);
		\end{tikzpicture}
		
	\end{center}
	form a biequivalence of 2-categories. Here $\Phi\tx{-}\bo{Acc}_\Psi$ is the 2-category whose objects are the $\Phi$-accessible $\V$-categories with $\Psi$-limits (\cite[Section~3.1]{LT22:virtual}), whose morphisms are the $\alpha$-continuous and $\Phi^+$-cocontinuous $\V$-functors, and whose 2-cells are $\V$-natural transformations. On the other hand, $\Psi^+\tx{-}\bo{Free}_\Phi$ is the 2-category of the $\Phi$-complete and $\Psi^+$-free $\V$-categories, $\Phi$-continuous and $\Psi^+$-cocontinuous $\V$-functors, and $\V$-natural transformations.
	
	By taking $\V=\bo{Set}$, $\Psi=\P$, and $\Phi$ to be the class for finite products one would then recover the duality of \cite{adamek2003duality} between finitary varieties and algebraic theories; while for a general (weakly) sound $\Phi$ one obtains the duality discussed in \cite{centazzo2002duality} building on the work of \cite{ABLR02:articolo}. For a general $\V$, but still $\Psi=\P$, a version of the corresponding duality was proved in the 2-categorical context \cite[Theorem~3.14]{di2023accessibility} by considering the 2-category $\V\tx{-}\bo{CAT}$ and using the notion of KZ doctrine and GU envelope (see also \cite[Remark~4.10]{LT22:limits}).
\end{obs}

\section{Examples}\label{examples-duality}

As a consequence of our Theorem~\ref{psi-duality}, each of the weakly sound classes $\Psi$ from Example~\ref{examplewsc} induces a duality theorem. Let us see some in particular.

When $\Psi=\P$, then $\Psi^+=\Q$ and we recover the Gabriel-Ulmer duality for locally $\alpha$-presentable categories which in the enriched context was first proved by Kelly:

\begin{teo}[\cite{Kel82:articolo}]
	The 2-functors 
	\begin{center}
		
		\begin{tikzpicture}[baseline=(current  bounding  box.south), scale=2]

			\node (f) at (0,0.4) {$\alpha\tx{-Lp}(-,\V)\colon \alpha\tx{-}\bo{Lp}$};
			\node (g) at (2.2,0.4) {$\bo{Lex}_\alpha\op\cocolon \tx{Lex}_\alpha(-,\V).$};
			
			\path[font=\scriptsize]

			([yshift=1.3pt]f.east) edge [->] node [above] {} ([yshift=1.3pt]g.west)
			([yshift=-1.3pt]f.east) edge [ <-] node [below] {} ([yshift=-1.3pt]g.west);
		\end{tikzpicture}
		
	\end{center}
	form a biequivalence of 2-categories.
\end{teo}

Indeed, that complete $\alpha$-accessible $\V$-categories are the same as locally $\alpha$-presentable $\V$-categories is standard (it follows for instance from~\cite[Theorem~4.22]{LT22:limits})) On the other hand, every small $\alpha$-complete  $\V$-category is Cauchy complete by Remark~\ref{absolute-alpha}. Thus $\Q\tx{-}\bo{Free}_\alpha=\bo{Lex}_\alpha$ coincides with the 2-category of $\alpha$-complete $\V$-categories.

When $\V=\bo{Set}$ and $\Psi$ is the class for connected limits, then $\Psi^+=\bo{Fam}$ is the class generated by discrete categories (the class for coproducts), and we recover Diers duality for locally $\alpha$-multipresentable categories. Below the 2-category $\bo{Fam}\tx{-}\bo{Lex}_\alpha$ has $\alpha$-complete coproduct cocompletions of small categories as objects, $\alpha$-continuous and coproduct-preserving functors as morphisms, and natural transformations as 2-cells.

\begin{teo}[\cite{Die80:articolo}]
	The 2-functors 
	\begin{center}
		
		\begin{tikzpicture}[baseline=(current  bounding  box.south), scale=2]
			
			\node (e) at (-1.5,0) {};
			\node (f) at (0,0) {$\alpha\tx{-Lmp}(-,\bo{Set})\colon \alpha\tx{-}\bo{Lmp}$};
			\node (g) at (3.1,0) {$\bo{Fam}\tx{-}\bo{Lex}_\alpha\op\cocolon \tx{Fam-Lex}_\alpha(-,\bo{Set}).$};
			
			\path[font=\scriptsize]

			([yshift=1.3pt]f.east) edge [->] node [above] {} ([yshift=1.3pt]g.west)
			([yshift=-1.3pt]f.east) edge [ <-] node [below] {} ([yshift=-1.3pt]g.west);
		\end{tikzpicture}
		
	\end{center}
	form a biequivalence of 2-categories.
\end{teo}

A 2-categorical version of this duality can be obtained by taking $\V=\bo{Cat}$ and $\Psi$ to be the weakly sound class consisting of the connected 2-limits described in Example~\ref{examplewsc}(11). 

When $\V$ is a generic base of enriched (as in our assumptions) and $\Psi=\emptyset$, so that $\Psi^+=\P$, we obtain a duality for $\alpha$-accessible $\V$-categories (with no existence of limit required): 

\begin{teo}\label{finacc-presh}
	The 2-functors 
	\begin{center}
		
		\begin{tikzpicture}[baseline=(current  bounding  box.south), scale=2]
			
			\node (f) at (0,0) {$\alpha\tx{-Acc}(-,\V)\colon\alpha\tx{-}\bo{Acc}$};
			\node (g) at (2.6,0) {$\P\tx{-}\bo{Lex}_\alpha\op\cocolon \P\tx{-Lex}_\alpha(-,\V).$};
			
			\path[font=\scriptsize]

			([yshift=1.3pt]f.east) edge [->] node [above] {} ([yshift=1.3pt]g.west)
			([yshift=-1.3pt]f.east) edge [ <-] node [below] {} ([yshift=-1.3pt]g.west);
		\end{tikzpicture}
		
	\end{center}
	form a biequivalence of 2-categories.
\end{teo}

Here the 2-category $\P\textnormal{-}\bo{Lex}_\alpha$ has presheaf $\V$-categories as objects, $\alpha$-continuous and cocontinuous $\V$-functors as morphisms, and $\V$-natural transformations as 2-cells.

In the enriched context this duality is new, while for $\alpha=\aleph_0$ and $\V=\bo{Set}$ it first appeared as \cite[Proposition~4.2.1]{MP89:libro}. Note that this is part of a larger adjunction between accessible categories with filtered colimits and Grothendieck topoi \cite{DiLib20:articolo}. In Section~\ref{Scott-V} we construct an enriched version of it.

When $\Psi$ is a locally small class of weights, then $\Psi^+$-cocompletions of small $\V$-categories coincide with the $\Psi$-accessible $\V$-categories of \cite[Section~3]{LT22:virtual} (by \cite[Proposition~3.10]{LT22:virtual}). It follows that $\Psi^+\textnormal{-}\bo{Free}_\alpha= \Psi\textnormal{-}\bo{Acc}_\alpha$ is the same as the 2-category of $\Psi$-accessible $\V$-categories with $\alpha$-small limits, $\alpha$-continuous and $\Psi^+$-cocontinuous $\V$-functors, and $\V$-natural transformations. Thus the duality becomes:

\begin{teo}\label{psi-small}
	Let $\Psi$ be a small and weakly sound class of weights; then the 2-functors 
	\begin{center}
		
		\begin{tikzpicture}[baseline=(current  bounding  box.south), scale=2]

			\node (f) at (0,0.4) 	{$\alpha\tx{-Acc}_\Psi(-,\V)\colon\alpha\tx{-}\bo{Acc}_\Psi$};
			\node (g) at (2.8,0.4) {$\Psi\tx{-}\bo{Acc}_\alpha\op\cocolon 	\Psi\textnormal{-}\tx{Acc}_\alpha(-,\V).$};
			
			\path[font=\scriptsize]
			
			([yshift=1.3pt]f.east) edge [->] node [above] {} ([yshift=1.3pt]g.west)
			([yshift=-1.3pt]f.east) edge [<-] node [below] {} ([yshift=-1.3pt]g.west);
		\end{tikzpicture}
		
	\end{center}
	form a biequivalence of 2-categories.
\end{teo}

This theorem has a couple of nice consequences. First, consider $\V=\bo{Set}$ and $\Psi=\bo{Fp}$ to be the weakly sound class for finite products. Then $\Psi^+=\bo{Sind}$ is the class of weights for sifted colimits (see for instance~\cite{ARV2010algebraic} for a characterization of these). The notion of $\bo{Fp}$-accessible category has been treated in \cite{AR11:articolo}, where they are called {\em generalized varieties}. Then, Theorem~\ref{psi-small} above implies that the 2-category of $\alpha$-accessible categories with finite products is dually equivalent to the 2-category of generalized varieties with $\alpha$-small limits. 

\begin{obs}
	The observation above has also an enriched generalization which applies whenever $\V$ is cartesian closed, since then the class of weights for finite products is weakly sound (Example~\ref{examplewsc}).
\end{obs}

Secondly, when $\Psi=\P_\alpha$ is the class of $\alpha$-small weights, a $\Psi$-accessible $\V$-category is just an $\alpha$-accessible $\V$-category. Thus $\Psi\tx{-}\bo{Acc}_\alpha= \alpha\tx{-}\bo{Acc}_\alpha$ coincides with the 2-category of $\alpha$-complete and $\alpha$-accessible $\V$-categories, $\alpha$-continuous and $\alpha$-flat-colimit preserving $\V$-functors, and $\V$-natural transformations. Therefore we obtain the following:

\begin{cor}\label{bi-inv}
	There is a biequivalence of 2-categories 
	$$ \alpha\tx{-}\bo{Acc}_\alpha \simeq (\alpha\tx{-}\bo{Acc}_\alpha)\op $$%
	which is induced by the 2-functor $\alpha\tx{-Acc}_\alpha(-,\V)$.
\end{cor}

Let $\Sigma=\alpha\tx{-Acc}_\alpha(-,\V)$ be the 2-functor involved in the duality; then in particular $\Sigma$ is a bi-involution: $\Sigma^2\simeq 1$. We can give a more direct way to describe the action of $\Sigma$ on objects and morphisms as follows.

Given an $\alpha$-accessible $\V$-category $\A$, left Kan extending along the inclusion induces an equivalence $[\A_\alpha,\V]\simeq \alpha\tx{-Acc}(\A,\V)$; if $\A$ is moreover $\alpha$-complete then, by \cite[Lemma~2.7]{LT21:articolo}, the equivalence restricts to $\alpha\tx{-Flat}(\A_\alpha,\V)\simeq \alpha\tx{-Acc}_\alpha(\A,\V)$. Thus $\Sigma\A\simeq\alpha\tx{-Flat}(\A_\alpha,\V)$, or equivalently: $$\Sigma(\alpha\tx{-Flat}(\C\op,\V))\simeq \alpha\tx{-Flat}(\C,\V).$$%
Similarly, given a morphism $F\colon\A\to\B$ in $\alpha\tx{-}\bo{Acc}_\alpha$ with inclusions $J\colon\A_\alpha\to\A$ and $H\colon\B_\alpha\to\B$, a few calculations show that the 2-functor $\Sigma$ acts as follows:
\begin{center}
	\begin{tikzpicture}[baseline=(current  bounding  box.south), scale=2]
		
		\node (a0) at (0,0.4) {$\alpha\tx{-Flat}(\B_\alpha,\V) $};
		\node (b0) at (1.8,0.4) {$\alpha\tx{-Flat}(\A_\alpha,\V)$};
		\node (c0) at (0.1,0) {$X$};
		\node (d0) at (1.8,0) {$(\tx{Lan}_HX)FJ.$};
		
		\path[font=\scriptsize]
		
		(a0) edge [->] node [above] {$\Sigma F$} (b0)
		(c0) edge [dashed, |->] node [below] {} (d0);
		
	\end{tikzpicture}	
\end{center}
The resulting $\V$-functor $(\tx{Lan}_HX)FJ$ is still $\alpha$-flat since $X$ is $\alpha$-flat and $F$ is $\alpha$-continuous and $\alpha$-flat-colimit preserving.

\section{Accessible $\V$-categories and $\V$-topoi}\label{Scott-V}

We have seen, as part of the duality in Theorem~\ref{finacc-presh}, that if $\A$ is a finitely accessible $\V$-category then $\aleph_0\tx{-Acc}(\A,\V)$ is equivalent to the presheaf $\V$-category $[\A_f,\V]$. Conversely, given any category of presheaves $[\C,\V]$ the category $\P\tx{-Lex}([\C,\V],\V)\simeq\tx{Flat}(\C\op,\V)$ is finitely accessible. In the ordinary case this is actually part of a wider adjunction between the 2-category of accessible categories with filtered colimits and that of Grothendieck topoi (see \cite{DiLib20:articolo}). We shall prove in this section that such an adjunction can be extended to the enriched setting by considering accessible $\V$-categories with flat colimits and the notion of $\V$-topos defined in \cite{GL12:articolo}.

We assume for simplicity that $\V$ is locally finitely presentable as a closed category; nonetheless, everything can be carried out, with the necessary modifications, in the infinitary case.

\begin{Def}[\cite{GL12:articolo}]
	We say that a $\V$-category $\E$ is a {\em $\V$-topos} if it is a left exact localization of a presheaf $\V$-category $[\C\op,\V]$. In other words, $\E$ is a $\V$-topos if there exists a fully faithful $J\colon\E\hookrightarrow[\C\op,\V]$ which has a lex left adjoint. Denote by $\V\tx{-}\bo{Top}$ the 2-category of $\V$-topoi, lex cocontinuous $\V$-functors, and $\V$-natural transformations. 
\end{Def}

\begin{obs}\label{lp-topos}
	By \cite[Proposition~2.6]{GL12:articolo} the small $\V$-category $\C$ above can be chosen to be a full subcategory of $\E$. That is, if $\E$ is a $\V$-topos then there exists a small dense full subcategory $H\colon\C\hookrightarrow\E$ for which the induced nerve $\V$-functor $J\colon\E\to[\C\op,\V]$ is fully faithful (by density of $\C$) and has a lex left adjoint. 
	
	Moreover, every $\V$-topos is locally presentable and the inclusion $J\colon\E\hookrightarrow [\C\op,\V]$ is an accessible embedding by \cite[Proposition~A.2]{GL12:articolo}.
\end{obs}

Denote by $\bo{Acc}^{\aleph_0}$ the 2-category of accessible $\V$-categories with filtered colimits, finitary (that is, flat colimits preserving) $\V$-functors, and $\V$-natural transformations. The same convention of Notation~\ref{notation} applies in this context for $\bo{Acc}^{\aleph_0}$ and $\V\tx{-}\bo{Top}$.

Then the result below can be seen as an extension of the Scott adjunction \cite{DiLib20:articolo} to the enriched context.

\begin{teo}\label{Vscott}
	The following
	\begin{center}
		
		\begin{tikzpicture}[baseline=(current  bounding  box.south), scale=2]

			\node (f) at (0,0.4) {$\tx{Acc}^{\aleph_0}(-,\V)\colon\bo{Acc}^{\aleph_0}$};
			\node (g) at (2.9,0.4) {$\V\tx{-}\bo{Top}\op\cocolon \V\tx{-Top}(-,\V)$};
			\node (h) at (1.4,0.45) {$\perp$};
			
			\path[font=\scriptsize]

			([yshift=-1.3pt]f.east) edge [->] node [above] {} ([yshift=-1.3pt]g.west)
			([yshift=1.3pt]f.east) edge [bend left,<-] node [below] {} ([yshift=1.3pt]g.west);
		\end{tikzpicture}
		
	\end{center}
	defines a 2-adjunction.
\end{teo} 
\begin{proof}
	The only non-trivial part is to show that the 2-functors are well-defined; indeed, the universal property that defines the 2-adjunction can be obtained by using the same arguments showing that the 2-functors of Theorem~\ref{psi-duality} form a 2-adjunction (see the discussion just above the statement). 
	
	Let $\E$ be a $\V$-topos; then $\V\tx{-Top}(\E,\V)$ is closed in $[\E,\V]$ under flat colimits (since these commute with finite limits), so that we only need to prove that it is accessible.
	Thanks to Remark~\ref{lp-topos}, the $\V$-topos $\E$ is locally presentable and we can take $H\colon\C\hookrightarrow \E$ small and such that the nerve $\V$-functor
	$$J\colon \E\hookrightarrow\P(\C)=[\C\op,\V],$$
	defined by $JE:=\E(H-,E)$, is fully faithful and has a lex left adjoint $L$. Consider now $\alpha$ such that $\E$ is locally $\alpha$-presentable; we will prove that the square below is a bipullback in $\V$-$\bo{CAT}$.
	\begin{center}
		\begin{tikzpicture}[baseline=(current  bounding  box.south), scale=2]
			
			\node (a0) at (0,0.8) {$\V\tx{-Top}(\E,\V)$};
			\node (b0) at (0,0) {$\alpha\tx{-Acc}(\E,\V)$};
			\node (c0) at (1.8,0.8) {$\V\tx{-Top}(\P(\C),\V)$};
			\node (d0) at (1.8,0) {$\alpha\tx{-Acc}(\P(\C),\V)$};
			\node (e0) at (0.3,0.55) {$\lrcorner$};
			
			\path[font=\scriptsize]
			
			(a0) edge [right hook->] node [above] {} (b0)
			(a0) edge [->] node [above] {$-\circ L$} (c0)
			(b0) edge [->] node [below] {$-\circ L$} (d0)
			(c0) edge [right hook->] node [below] {} (d0);
		\end{tikzpicture}	
	\end{center}
	For that, it is enough to notice that a $\V$-functor $F\colon\E\to\V$ is lex-cocontinuous if and only if $FL$ is. One direction is clear (since $L$ is lex-cocontinuous), for the other assume that $FL$ is lex-cocontinuous, then $F\cong FLJ$ is lex because $J$ preserves all limits. Moreover for any diagram $D\colon\D\to\E$ and weight $M\colon\D\op\to\V$ we obtain:
	\begin{equation*}
		\begin{split}
			F(M*D)&\cong F(M* LJD)\\
			&\cong FL(M* JD)\\
			&\cong M*FLJD\\
			&\cong M*FD\\
		\end{split}
	\end{equation*}
	so that $F$ is cocontinuous. Now, Theorem~\ref{finacc-presh} provides an equivalence of categories $\V\tx{-Top}(\P(\C),\V)\simeq\tx{Flat}(\C,\V)$, while $\alpha\tx{-Acc}(\E,\V)\simeq [\E_\alpha,\V]$ and $\alpha\tx{-Acc}(\P(\C),\V)\simeq [\P(\C)_\alpha,\V]$ since $\E$ and $\P(\C)$ are locally $\alpha$-presentable. Moreover, the inclusion $\V\tx{-Top}(\E,\V)\hookrightarrow \alpha\tx{-Acc}(\P(\C),\V)$ preserves flat-colimits (since these commute with all colimits and finite limits in $\V$), and the $\V$-functor $-\circ L\colon \alpha\tx{-Acc}(\E,\V)\to \alpha\tx{-Acc}(\P(\C),\V)$ is cocontinuous since colimits are computed pointwise in both $\V$-categories.
	
	It follows that $\V\tx{-Top}(\E,\V)$ can be seen as a bipullback of accessible $\V$-categories along accessible $\V$-functors; thus it is itself accessible by~\cite[Theorem~5.5]{LT22:virtual}. 
	The fact that $\V\tx{-Top}(F,\V)$ is finitary whenever $F\colon\E\to\F$ is lex cocontinuous, follows from the fact that flat colimits are computed pointwise in $\V\tx{-Top}(\E,\V)$ and $\V\tx{-Top}(\F,\V)$, and precomposition with $F$ is continuous and cocontinuous in the presheaf $\V$-categories.
	
	Conversely, let $\A$ be an accessible $\V$-category with flat colimits and consider $\alpha$ such that $\A$ is $\alpha$-accessible. Denote by 
	$$W\colon \tx{Acc}^{\aleph_0}(\A,\V) \hookrightarrow\alpha\tx{-Acc}(\A,\V) \simeq[\A_{\alpha},\V]$$
	the inclusion; this preserves all colimits and all finite limits since these are computed pointwise in $\alpha\tx{-Acc}(\A,\V)$ and filtered colimits commute with all colimits and finite limits in $\V$. In particular $\tx{Acc}^{\aleph_0}(\A,\V)$ is cocomplete and lex. 
	
	Now define $\B:=\tx{Ind}(\A_\alpha)$ to be the free cocompletion of $\A_\alpha$ under flat colimits, so that $\B$ is finitely accessible; then we have induced $\V$-functors $S\colon\A\to\B$, which preserves $\alpha$-flat colimits and extends the inclusion $\A_\alpha\subseteq\B$ to $\A$, and $T\colon\B\to\A$, which preserves flat colimits and extends the inclusion $\A_\alpha\subseteq\A$ to $\B$. By construction they satisfy $TS\cong 1_\A$. Arguing as in the chain of isomorphisms above it follows that a $\V$-functor $F\colon\A\to\V$ preserves flat colimits if and only if $FT$ preserves them. As a consequence we can see $\tx{Acc}^{\aleph_0}(\A,\V)$ as the bipullback below in $\V$-$\bo{CAT}$.
	\begin{center}
		\begin{tikzpicture}[baseline=(current  bounding  box.south), scale=2]
			
			\node (a0) at (0,0.8) {$\tx{Acc}^{\aleph_0}(\A,\V)$};
			\node (b0) at (1.8,0.8) {$\tx{Acc}^{\aleph_0}(\B,\V)$};
			\node (c0) at (0,0) {$\alpha\tx{-Acc}(\A,\V)$};
			\node (d0) at (1.8,0) {$\alpha\tx{-Acc}(\B,\V)$};
			\node (e0) at (0.3,0.55) {$\lrcorner$};
			
			\path[font=\scriptsize]
			
			(a0) edge [->] node [above] {$-\circ T$} (b0)
			(a0) edge [right hook->] node [left] {} (c0)
			(b0) edge [right hook->] node [right] {} (d0)
			(c0) edge [->] node [below] {$-\circ T$} (d0);
		\end{tikzpicture}	
	\end{center}
	Since $\tx{Acc}^{\aleph_0}(\B,\V)\simeq[\A_{\alpha},\V]\simeq\alpha\tx{-Acc}(\A,\V)$ and 
	$\alpha\tx{-Acc}(\B,\V)\simeq[\B_\alpha,\V]$ are locally presentable and the $\V$-functors between them are accessible (by arguments similar to those used above), it follows that $\E:=\tx{Acc}^{\aleph_0}(\A,\V)$ is accessible as well (again by \cite[Theorem~5.5]{LT22:virtual}). Moreover since $\E$ is cocomplete, it is also locally presentable.
	
	Consider then a small dense full subcategory $H\colon\C\hookrightarrow\E$ of $\E$ closed under finite weighted limits (a small dense subcategory exists by local presentability, then take the closure of this in $\E$ under finite limits). Then, the nerve $\V$-functor $$K\colon \E\hookrightarrow[\C\op,\V],$$ 
	defined by $KE=\E(H-,E)$, is fully faithful and, by local presentability of $\E$ and $[\C\op,\V]$, has a left adjoint given by $\tx{Lan}_YH\colon[\C\op,\V]\to\E$. To conclude that $\E$ is a $\V$-topos then it is enough to prove that $\tx{Lan}_YH$ is lex.
	Consider the composite $W(\tx{Lan}_YH)$, which is cocontinuous since $W$ and $\tx{Lan}_YH$ are. Since this also restricts to $WH$, it follows that $$W(\tx{Lan}_YH)\cong \tx{Lan}_Y(WH).$$ 
	But $\tx{Lan}_Y(WH)\colon [\C\op,\V]\to [\A_\alpha,\V]$ is lex by \cite[Proposition~2.4(4)]{GL12:articolo} since $[\A_\alpha,\V]$ is a $\V$-topos and $\V$-topoi satisfy all the (equivalent) conditions of \cite[Proposition~2.4]{GL12:articolo} by \cite[Proposition~2.6]{GL12:articolo}. Thus also $\tx{Lan}_YH$ is lex, because $W$ is fully faithful and preserves finite limits. As a consequence $\E$ is a left exact localization of $[\C\op,\V]$ and hence a $\V$-topos. The fact that $\tx{Acc}^{\aleph_0}(F,\V)$ is lex cocontinuous whenever $F\colon\A\to\B$ is finitary, follows from the fact that finite limits and colimits are computed pointwise in $\tx{Acc}^{\aleph_0}(\A,\V)$ and $\tx{Acc}^{\aleph_0}(\B,\V)$, and pre-composition with $F$ is continuous and cocontinuous in the presheaf $\V$-categories.
\end{proof}

\begin{obs}
	We expect that the duality above could further generalized in two different directions by considering weakly sound classes: \begin{enumerate}[align=parleft,left=0pt]\setlength\itemsep{0.25em}
		\item Given a weakly sound class of weights $\Psi$, replace $\V\tx{-}\bo{Top}$ with the 2-category whose objects are the accessible and left exact localizations of finitely complete $\Psi$-free $\V$-categories; that is, of the form ${\Psi^+}\B$, for some small $\B$. Morphisms between them are lex and ${\Psi^+}$-cocontinuous $\V$-functors, and 2-cells are 2-natural transformations. On the other hand, instead of $\bo{Acc}^{\aleph_0}$, consider the 2-category of accessible $\V$-categories with $\Psi$-limits and flat colimits, $\Psi$-continuous and finitary $\V$-functors, and $\V$-natural transformations. Then one recovers the adjunction of Theorem~\ref{Vscott} above by taking $\Psi=\emptyset$ and ${\Psi^+}=\P$. 
		\item Given a locally small weakly sound class of weights $\Phi$. On one hand, we could replace the $\V$-topoi above with {\em $\Phi$-topoi}: accessible and reflective subcategories of a presheaf $\V$-category with a $\Phi$-continuous left adjoint; morphisms between them are cocontinuous and $\Phi$-continuous $\V$-functors. On the other hand one considers accessible $\V$-categories with $\Phi$-flat colimits and $\V$-functors which preserve them. Then one recovers the adjunction of Theorem~\ref{Vscott} above by taking $\Phi$ to be the class of finite weights.
	\end{enumerate}
\end{obs}


\end{document}